\documentclass[reqno]{amsart}
\usepackage{amsmath, amssymb, amsfonts}
\newtheorem{theorem}{Theorem}

\newtheorem{corollary}{Corollary}
\newtheorem{lemma}{Lemma}

\theoremstyle{remark}
\newtheorem{remark}{Remark}
\theoremstyle{definition}
\newtheorem{define}{Definition}
\newtheorem{example}{Example}

\newcommand{\CC}{\mathbb C}

\begin{document}

\title[Infinitesimal CR automorphisms and stability groups of models]
{Infinitesimal CR automorphisms and stability groups of infinite type 
models in $\mathbb C^2$}

\author{Atsushi Hayashimoto and Ninh Van Thu}

\thanks{ The research of the second author was supported in part by 
an NRF grant 2011-0030044 (SRC-GAIA) of the Ministry of Education, 
The Republic of Korea.}
\address{Nagano National College of Technology, 716 Tokuma, Nagano 318-8550}
\email{atsushi@nagano-nct.ac.jp}
\address{(Permanent Address) Department of Mathematics, Vietnam
National University at Hanoi, 334 Nguyen Trai str., Hanoi, Vietnam}
\email{thunv@vnu.edu.vn}
\address{(Current Address) Center for Geometry and its Applications,
 Pohang University of Science and Technology,  Pohang 790-784, 
The Republic of Korea}
\email{thunv@postech.ac.kr}

\subjclass[2000]{Primary 32M05; Secondary 32H02, 32H50, 32T25.}
\keywords{Holomorphic vector field, automorphism group, 
real hypersurface, infinite type point.}
\begin{abstract}
The purpose of this paper is to give explicit descriptions for stability groups 
of real rigid hypersurfaces of infinite type  in $\mathbb C^2$.
The decompositions of infinitesimal CR automorphisms are also given. 
\end{abstract}
\maketitle

\section{Introduction}
Let $M$ be a $\mathcal{C}^\infty$-smooth real hypersurface 
in $\mathbb C^n$ and $p\in M$. 
We denote by $\mathrm{Aut}(M)$ the CR automorphism group of $M$, 
by $\mathrm{Aut}(M,p)$ the stability group of $M$, that is, those germs at $p$ 
of biholomorphisms mapping $M$ into itself and fixing $p$, 
and by $\mathfrak{aut}(M,p)$ the set of germs of holomorphic vector fields 
in $\mathbb C^n$ at $p$ whose real part is tangent to $M$. 
We call this set the Lie algebra of infinitesimal CR automorphisms. 
We also denote by 
$\mathfrak{aut}_0(M,p):=\{H\in \mathfrak{aut}(M,p) \colon H(p)=0\}$. 

For a real hypersurface in $\mathbb C^n$, 
the stability group and the Lie algebra of infinitesimal 
CR automorphisms are not easy to describe explicitly; 
besides, it is unknown in most cases. 
But, the study of $\mathrm{Aut}(M,p)$ and $\mathfrak{aut}(M,p)$ of 
special types of hypersurfaces is given in 
\cite{Chern, Ezhov, Ezhov1, Kolar4, Kolar3, Kolar2, Kolar1, Kolar5, Stanton1, Stanton2}. 
For instance, explicit forms of the stability groups of models 
(see detailed definition in \cite{Kolar4, Kolar5}) have been obtained 
in \cite{Ezhov1, Kolar4, Kolar3, Kolar5}. However, these results are 
known for Levi nondegenerate hypersurfaces 
or more generally for Levi degenerate 
hypersurfaces of finite type in the sense of D'Angelo (cf. \cite{D}).

In this article, we give explicit descriptions for the Lie algebra of infinitesimal 
CR automorphisms and for the stability group of an infinite type model  
$(M_P,0)$ in $\mathbb C^2$ which is defined by
$$
M_P:=\{(z_1,z_2)\in \mathbb C^2\colon \mathrm{Re}~z_1+P(z_2)=0\},
$$
where $P$ is a nonzero germ of a real-valued $\mathcal{C}^\infty$-smooth 
function at $0$ vanishing to infinite order at $z_2=0$. 

To state these results more precisely, we establish some notation. 
Denote by $\mathrm{G_2}(M_P,0)$ the set of all 
CR automorphisms of $M_P$ defined by
$$
(z_1,z_2)\mapsto (z_1,g_2(z_2)),
$$
for some holomorphic function $g_2$ with $g_2(0)=0$ and $|{g_2}'(0)|=1$ 
defined on a neighborhood of the origin in $\mathbb C$ satisfying that 
$P(g_2(z_2))\equiv P(z_2)$. 
Also denote by $\Delta_{\epsilon_0}$ a disc with center at the origin 
and radius $\epsilon_0$ and by $\Delta_{\epsilon_0}^*$ a punctured disc 
$\Delta_{\epsilon_0} \backslash \{0\}$. 

Let $P:\Delta_{\epsilon_0} \to \mathbb R$ be 
a $\mathcal{C}^\infty$-smooth function. 
Let us denote by 
$S_\infty(P)=\{z\in \Delta_{\epsilon_0}\colon \nu_z(P)=+\infty\}$, 
where $\nu_z(P)$ is the vanishing order of $P(z+\zeta)-P(z)$ at $\zeta=0$, 
and by $P_\infty(M_P)$ the set of all points of infinite type in $M_P$.
\begin{remark} 
It is not hard to see that 
$P_\infty(M_P)=\{(it-P(z_2),z_2) \colon t\in \mathbb R, z_2\in S_\infty(P)\}.$
\end{remark}
\begin{remark} In the case that $P\not \equiv 0$, $\mathrm{G_2}(M_P,0)$ 
contains only CR automorphisms of $M_P$ defined by
$$ 
(z_1,z_2)\mapsto (z_1,g_2(z_2)),
$$
where $g_2$ is a conformal map with $g_2(0)=0$ satisfying $P(g_2(z_2))\equiv P(z_2)$ and either
 ${g_2}'(0)= e^{2\pi i p/q}~(p,q\in \mathbb Z)$ and ${g_2}^q=\mathrm{id}$ or ${g_2}'(0)= e^{2\pi i \theta}$ for some $\theta \in\mathbb R\setminus \mathbb Q$ (see cf. Lemma \ref{lemma1} in \S~2 and Lemmas \ref{parabolic} and \ref{elliptic} in \S~3).
\end{remark}

The first aim of this paper is to prove the following two theorems, 
which give a decomposition of the infinitesimal CR automorphisms 
and an explicit description for stability groups of infinite type models.
In what follows, all functions, mappings, hypersurfaces, etc are understood 
to be germs at the reference points and we will not refer it 
if there is no confusions. 

\begin{theorem}\label{T1}
Let $(M_P,0)$ be a real $\mathcal{C}^\infty$-smooth 
hypersurface defined by the equation
$\rho(z) := \rho(z_1,z_2)=\mathrm{Re}~z_1+P(z_2)=0$, where 
$P$ is a $\mathcal{C}^\infty$-smooth function on a neighborhood of 
the origin in $\mathbb C$ satisfying the conditions: 
\begin{itemize}
\item[(i)] $P(z_2)\not \equiv 0$ on a neighborhood of $z_2= 0$, and
\item[(ii)] The connected component of $0$ in $S_\infty(P)$ is $\{0\}$. 
\end{itemize}
Then the following assertions hold: 
\begin{itemize}
\item[(a)] The Lie algebra $\mathfrak{g}=\mathfrak{aut}(M_P,0)$ 
admits the decomposition
$$
\mathfrak{g}=\mathfrak{g}_{-1}\oplus \mathfrak{aut}_0(M_P,0),
$$
where 
$\mathfrak{g}_{-1}=\{i\beta \partial_{z_1}\colon \beta\in \mathbb R\}$.
\item[(b)] If $\mathfrak{aut}_0(M_P,0)$ is trivial, then
$$
\mathrm{Aut}(M_P,0)=\mathrm{G_2}(M_P,0).
$$
\end{itemize}
\end{theorem}

\begin{remark}
The condition $\mathrm{(ii)}$ simply tells us that $M_P$ is of infinite type. 
Moreover,  the connected component of $0$ in $P_\infty(M_P)$ is 
the set $\{(it,0)\colon t\in \mathbb R\}$, which plays 
a key role in the proof of this theorem. 
\end{remark}

In the case that the connected component of $0$ in 
$S_\infty(P)$ is not $\{0\}$ 
such as $M_P$ is tubular, we have the following theorem.
\begin{theorem}\label{T11}
Let $\tilde P$ be a $\mathcal{C}^\infty$-smooth function defined 
on a neighborhood of $0$ in $\mathbb C$ satisfying:
\begin{itemize}
\item[(i)] $\tilde P(x)\not \equiv 0$ on a neighborhood of 
$x= 0$ in $\mathbb R$, and
\item[(ii)] The connected component of $0$ in $S_\infty(\tilde P)$ is $\{0\}$.
\end{itemize}
Denote by $P$ a function defined by setting 
$P(z_2):=\tilde P(\mathrm{Re}~z_2)$.
Then the following assertions hold:
\begin{itemize}
\item[(a)]$\mathfrak{aut}_0(M_P,0)=0$ and 
the Lie algebra $\mathfrak{g}=\mathfrak{aut}(M_P,0)$ 
admits the decomposition
$$
\mathfrak{g}=\mathfrak{g}_{-1}\oplus\mathfrak{g}_{0},
$$
where $\mathfrak{g}_{-1}=\{i\beta \partial_{z_1}\colon \beta\in \mathbb R\}$ 
and $\mathfrak{g}_{0}=\{i\beta \partial_{z_2}\colon \beta\in \mathbb R\}$.
\item[(b)] $\mathrm{Aut}(M_P,0)=\{id\}$.
\item[(c)] If $S_\infty(\tilde P)=\{0\}$, 
then $\mathrm{Aut}(M_P)=\mathrm{T^1}(M_P)\oplus\mathrm{T^2}(M_P)
=\{(z_1,z_2)\mapsto (z_1+it,z_2+is)\colon t,s\in\mathbb R\}$, 
where $\mathrm{T^1}(M_P)=\{(z_1,z_2)\mapsto (z_1+it,z_2)\colon t\in\mathbb R\}$ 
and $\mathrm{T^2}(M_P)=\{(z_1,z_2)\mapsto (z_1,z_2+it)\colon t\in\mathbb R\}$. 
\end{itemize}
\end{theorem}

These theorems shows that the special conditions of 
defining functions determine 
the forms of holomorphic vector fields. Conversely, the second aim of this paper is to show that holomorphic vector fields determine the form of defining functions. This is, in some sense, the converse of Example~\ref{ex2} in \S~6 holds generally. Namely, we prove the following.
\begin{theorem}\label{T3}
Let  $(M_P,0)$ be a $\mathcal{C}^\infty$-smooth hypersurface 
defined by the equation
$\rho(z) := \rho(z_1,z_2)=\mathrm{Re}~z_1+P(z_2)=0$,
satisfying the conditions:
\begin{itemize}
\item[(i)] The connected component of $z_2=0$ in the zero set of $P$ is $\{0\}$;
\item[(ii)] $P$ vanishes to infinite order at $z_2=0$. 
\end{itemize}
Then any holomorphic vector field vanishing at the origin tangent to $(M_P,0)$ 
is either identically zero, or, after a change of variable in $z_2$, 
of the form $i \beta z_2\partial_{z_2}$ for some non-zero real number 
$\beta$, in which case $M_P$ is rotationally symmetric, i.e. 
$P(z_2)=P(|z_2|)$.
\end{theorem}

The organization of this article is the following. In \S~2, 
we prove three lemmas which we use in the proof of theorems.    
In \S~3, we give a description of stability groups and proofs of 
Theorems \ref{T1} and \ref{T11} are given in \S~4. 
In \S~5, we prove Theorem~\ref{T3} and lemmas needed to prove it. 
In \S~6, we introduce some examples. 
Finally, two theorems are presented in Appendix A.

\section{Preliminaries}
In this section, we shall recall some definitions and introduce three lemmas 
which are used to prove Theorems \ref{T1} and \ref{T11}.
\begin{define} 
Let $g_1,g_2$ be two conformal maps with $g_1(0)=g_2(0)=0$. 
We say that $g_1$ and $g_2$ are holomorphically locally conjugated 
if there exists a biholomorphism $\varphi$ with $\varphi(0)=0$ such that
$$
g_1\equiv\varphi^{-1}\circ g_2\circ \varphi.
$$ 
\end{define}
\begin{define}
Let $g$ be a conformal map with $g(0)=0$. Then
\begin{itemize}
\item[(i)] if $g'(0)=1$, we say that $g$ is tangent to the identity;
 \item[(ii)] if $g'(0)=e^{2\pi i p/q},~p,q\in \mathbb Z$, we say that $g$ is parabolic;
 \item[(iii)] if $g'(0)=e^{2\pi i \theta}$ for some 
            $\theta \in \mathbb R\setminus\mathbb Q$, 
we say that $g$ is elliptic.
\end{itemize}
\end{define}
The following lemma is a slight generalization of \cite[Lemma $2$]{Ninh1}. 

\begin{lemma}\label{lemma2} 
Let $P$ be a $\mathcal{C}^\infty$-smooth function 
on $\Delta_{\epsilon_0}~(\epsilon_0>0)$ satisfying  $\nu_0(P)=+\infty$ and
 $P(z)\not \equiv 0$. 
Suppose that there exists a conformal map $g$ on $\Delta_{\epsilon_0}$ 
with $g(0)=0$ such that
$$
 P(g(z))=\big(\beta+o(1)\big)P(z), \; z\in \Delta_{\epsilon_0}
 $$ 
for some $\beta\in \mathbb R^*$. Then $|g'(0)|=1$.
\end{lemma}
\begin{proof}
Suppose that there exists a conformal map $g$ with $g(0)=0$ 
and a $\beta\in \mathbb R^*$ such that $P(g(z))=\big(\beta+o(1)\big)P(z)$ 
holds for  
$z \in \Delta_{\epsilon_0}$. 
Then, we have 
$$
P(g(z))=\big(\beta+\gamma(z)\big) P(z), \; z\in \Delta_{\epsilon_0},
$$
where $\gamma$ is a function defined on $\Delta_{\epsilon_0}$ with $\gamma(z)\to 0$ 
as $z\to 0$, which implies that there exists $\delta_0>0$ such that $|\gamma(z)|<|\beta|/2$ 
for any $z\in \Delta_{\delta_0}$. 
We consider the following cases.

\noindent
{\bf Case 1.} $0<|g'(0)|<1$. 
In this case, we can choose $\delta_0$ and $\alpha$ with $0<\delta_0<\epsilon_0$ 
and $|g'(0)|<\alpha <1$ such that $|g(z)|\leq\alpha |z|$ for all $z$ in $ \Delta_{\delta_0}$. 
Fix a point $z_0\in \Delta_{\delta_0}^* $ with $P(z_0)\ne 0$. 
Then, for each positive integer $n$, we get
\begin{equation}\label{eeq1}
\begin{split}
|P(g^n(z_0))|&=|\big(\beta+\gamma(g^{n-1}(z_0))\big)| |P(g^{n-1}(z_0))|=\cdots \\
            &  =|\big(\beta+\gamma(g^{n-1}(z_0))\big)| \cdots |\big(\beta+\gamma(z_0)\big)| |P(z_0)|\\
          &\geq \big(|\beta|-|\gamma(g^{n-1}(z_0))|\big) \cdots \big(|\beta|-|\gamma(z_0)|\big)|P(z_0)|\\
                              &\geq  \big(|\beta|/2\big)^n|P(z_0)|,
\end{split}
\end{equation}
where $g^n$ denotes the composition of $g$ with itself $n$ times. 
Moreover, since $0<\alpha<1$, there exists a positive integer $m_0$ 
such that $|\alpha^{m_0}|<|\beta|/2$. 
Notice that $0<|g^n(z_0)|\leq \alpha^{n}|z_0|$ for any $n\in \mathbb N$. 
Then it follows from (\ref{eeq1}) that
\begin{equation*}
\begin{split}
\frac{|P(g^{n}(z_0))|}{|g^{n}(z_0)|^{m_0}}&\geq\frac{|P((z_0)|}{|z_0|^{m_0}}\Big(\frac{|\beta|/2}
{\alpha^{m_0}}\Big)^n.
\end{split}
\end{equation*}
This yields that $|P(g^{n}(z_0))|/|g^{n}(z_0)|^{m_0} \to +\infty$ as $n\to \infty$, 
which contradicts the fact that $P$ vanishes to infinite order at $0$. 

\noindent
{\bf Case 2.} $1<|g'(0)|$. Since $P(g(z))=(\beta+o(1))P(z)$ for all $z\in \Delta_{\epsilon_0}$, 
it follows that $P(g^{-1}(z))=(1/\beta+o(1))P(z)$ for all $z\in \Delta_{\epsilon_0}$, 
which is impossible because of Case $1$. 

Altogether, $|g'(0)|=1$, and the proof is thus complete.
\end{proof}
\begin{lemma}\label{lemma-3} Let $f:[-r,r]\to \mathbb R~(r>0)$ be a continuous function satisfying $f(0)=0$ and $f\not \equiv 0$. If $\beta$ is a real number such that 
$$
f(t+\beta f(t))=f(t)
$$
for every $t\in [-r,r]$ with $t+\beta f(t)\in [-r,r]$, then $\beta=0$.
\end{lemma}
\begin{proof}
Suppose, to derive a contradiction, that there exists a $\beta\ne 0$ such that $f(t+\beta f(t))=f(t)$ for every $t\in [-r,r]$ with $t+\beta f(t)\in [-r,r]$. Then we have
\begin{equation}\label{2014eq-1}
\begin{split}
f(t)&=f(t+\beta f(t))=f\big(t+\beta f(t)+\beta f(t+\beta f(t))\big)\\
    &=f\big(t+2\beta f(t)\big)=\cdots=f(t+m\beta f(t))
\end{split}
\end{equation}
for every $m\in \mathbb N$ and for every $t\in [-r,r]$ with $t+m\beta f(t)\in [-r,r]$. 

Let $t_0\in [-r,r]$ be such that $f(t_0)\ne 0$. Then since $f$ is uniformly continous on $[-r,r]$, for every $\epsilon>0$ there exists $\delta>0$ such that for every $t_1,t_2\in [-r,r]$ with $|t_1-t_2|<\delta$, we have that $|f(t_1)-f(t_2)|<\epsilon/2$. On the other hand, since $f(t)\to 0$ as $t\to 0$ and since $f\not \equiv 0$, one can find $t\in [-\delta/2,\delta/2]$ such that $|\beta f(t)|<\delta$ and $0<|f(t)|<\epsilon/2$. Therefore, there exists an integer $m$ such that $|t+m\beta f(t)-t_0|<\delta$, and thus by (\ref{2014eq-1}) one has  
$$
|f(t_0)|\leq |f(t+m\beta f(t))-f(t_0)|+|f(t+m\beta f(t))|<\epsilon/2+ |f(t)|<\epsilon/2+\epsilon/2=\epsilon.
$$
This implies that $f(t_0)=0$, which is a contradiction. Hence, the proof is complete.
\end{proof}

\begin{lemma}\label{lemma1} 
Let $P$ be a nonzero $\mathcal{C}^\infty$-smooth function with 
$P(0)=0$ and let $g$ be a conformal map satisfying $g(0)=0$, $|g'(0)|=1$, 
and $g\ne \mathrm{id}$. 
If there exists a real number $\delta\in \mathbb R^*$ such that  
$P(g(z))\equiv \delta P(z)$, 
then $\delta=1$. 
Moreover, we have either 
$g'(0)= e^{2\pi i p/q}~(p,q\in \mathbb Z)$ and $g^q=\mathrm{id}$ 
or $g'(0)=e^{2\pi i \theta}$ for some $\theta \in \mathbb R\setminus\mathbb Q$.
\end{lemma}
\begin{proof} Replacing $g$ by its inverse if necessary, one can assume that $|\delta|\geq 1$. Now we divide the proof into three cases as follows:

\noindent
{\bf Case 1. $g'(0)=1$.} As a consequence of the Leau-Fatou flower theorem 
(cf. Theorem \ref{flower} in Appendix A.1), 
there exists a point $z$ in a small neighborhood of the origin with $P(z)\ne 0$ 
such that $g^n(z)\to 0$ as $n\to \infty$. 
Since $P(g^n(z))=(\delta)^n P(z)$ and $\lim_{n\to +\infty} P(g^n(z))=P(0)=0$, 
we have $0<|\delta|<1$, which is a contradiction.    

\noindent
{\bf Case 2. $\lambda:= g'(0)=e^{2\pi i p/q}~(p,q\in \mathbb Z)$.} 
Suppose that $g^q=\mathrm{id}$, then by \cite[Prop. $3.2$]{Abate}, 
there exists $z$ in a small neighborhood of $0$ satisfying $P(z)\ne 0$ such that 
the orbit $\{g^n(z)\}$ is contained a relativity compact subset of  
some punctured neighborhood.  
Therefore, by assumption $P(g(z))\equiv \delta P(z)$, 
the sequence $\{\delta^n\}$ must be convergent. 
This means that $\delta=1$. 
In the case of $g^q\ne\mathrm{id}$, we have $g^q(z)=z+\cdots$ and $P(g^q(z))\equiv \delta^q P(z)$. 
This is absurd because of Case $1$ with $g$ being 
replaced by $g^q$.   
 
\noindent
{\bf Case 3. $\lambda:= g'(0)=e^{2\pi i \theta}~(\theta \not \in \mathbb Q)$.} 
By \cite[Proposition~$4.4$]{Abate}, we may assume that there exists $z$ 
in a small neighborhood of $0$ satisfying $P(z)\ne 0$ such that 
the orbit $\{g^n(z)\}$ is contained a relativity compact subset of 
some punctured neighborhood. 
Therefore, the same argument as in Case $2$ shows that $\delta=1$.
Altogether, the proof is complete.
\end{proof}

\section{Explicit description for $\mathrm{G_2}(M_P,0)$}
In this section, we are going to give an explicit description 
for the subgroup $\mathrm{G_2}(M_P,0)$ of the stability group of $M_P$. 
By virtue of Lemma \ref{lemma1}, $\mathrm{G_2}(M_P,0)$ contains only 
CR automorphisms of $M_P$ defined by
$$ 
(z_1,z_2)\mapsto (z_1,g_2(z_2)),
$$
where $g_2$ is either parabolic or elliptic. 
Conversely, given either a parabolic $g$ with $g^q=\mathrm{id}$ 
for some positive integer $q$ or an elliptic $g$, 
we shall show that there exist some infinite type models $(M_P,0)$ 
such that the mapping $(z_1,z_2)\mapsto (z_1,g(z_2))$ belongs to 
$\mathrm{G_2}(M_P,0)$. 

First of all, we need the following lemma.
\begin{lemma} 
If $P(e^{2\pi i\theta}z)\equiv P(z)$ for some 
$\theta \in\mathbb R\setminus \mathbb{Q}$, 
then $P(z)\equiv P(|z|)$, i.e., $P$ is rotational.
\end{lemma}
\begin{proof}
We note that $P(e^{2\pi n i\theta}z)\equiv P(z)$ for any $n\in \mathbb N$ 
and $\overline{\{e^{2\pi n i\theta}z\colon n\in \mathbb N\}}=\mathbb S_{|z|}$, 
where $\mathbb S_r:=\{z\in \mathbb C\colon |z|=r\}$ for $r>0$. 
Therefore, because of the continuity of P, we conclude that $P(z)\equiv P(|z|)$. 
\end{proof}
\subsection{The Parabolic Case}
\begin{lemma}\label{parabolic} 
Let $g(z)=e^{2\pi i p/q}z+\cdots$ be a conformal map  
with $\lambda=e^{2\pi i p/q}$ being a primitive root of the unity. 
If $g^q=\mathrm{id}$, then there exists an infinite type model $M_P$ such that 
$(z_1,z_2)\mapsto (z_1,g^j(z_2))$ belongs to $\mathrm{G_2}(M_P,0)$ 
for every $j=1,2,\ldots, q-1$.
\end{lemma}
\begin{proof}
Suppose that $g(z)=e^{2\pi i p/q}z+\cdots$ is a conformal map 
such that $\lambda=e^{2\pi i p/q}$ is a primitive root of 
the unity satisfying $g^q=\mathrm{id}$. 
It is known that $g$ is holomorphically locally conjugated to $h(z)=\lambda z$ 
(cf. \cite[Proposition~$3.2$]{Abate}).
Let $\tilde P$ be a $\mathrm{C}^\infty$-smooth function with 
$\nu_0(\tilde P)=+\infty$. 
Define a $\mathrm{C}^\infty$-smooth function by setting
$$  
P(z):=\tilde P(z)+\tilde P(g(z))+\cdots+\tilde P(g^{q-1}(z)).
$$
Then it is easy to see that $P(g(z))\equiv P(z)$. 
Thus $f_j(z_1, z_2)=(z_1, g^{j}(z_2)) \in \mathrm{G_2}(M_P,0), j=1, \dots, q-1,$ 
are biholomorphic.  
\end{proof}
\begin{remark} 
In the case of $g^q\ne\mathrm{id}$. 
We have $g^d(z)=z+\cdots$, and therefore $P(z+\cdots)=P(g^q(z))=P(z)$. 
It follows from Lemma \ref{lemma1} that there is no infinite type model $M_P$ 
satisfying $P\not \equiv 0$ on some petal 
such that $(z_1,z_2)\mapsto (z_1,g(z_2))$ belongs to $\mathrm{G_2}(M_P,0)$. 
\end{remark}
\subsection{The Elliptic Cases}
\begin{lemma}\label{elliptic} 
Let $g(z)=e^{2\pi i\theta }z+\cdots$ be a conformal map with 
$\theta \not \in \mathbb{Q}$. 
Then there exists an infinite type model $M_P$ such that 
$(z_1,z_2)\mapsto (z_1,g(z_2))$ 
belongs to $\mathrm{G_2}(M_P,0)$. 
Moreover, $M_P$ is biholomorphically equivalent to a rotationally symmetric model 
$M_{\tilde P}$.
\end{lemma}
\begin{proof}
Suppose that $g(z)=e^{2\pi i\theta }z+\cdots$ is a conformal map 
with $\theta \not \in \mathbb{Q}$. 
Then it is known that $g$ is formally locally conjugated to 
$R_\theta(z)= e^{2\pi i\theta} z$ (cf. \cite[Proposition~$4.4$]{Abate}), i.e., 
there exists a formally conformal map $\varphi$ at $0$ with $\varphi(0)=0$ such that
$$
g=\varphi^{-1}\circ R_\theta\circ \varphi.
$$
Let $\tilde P$ be a rotational $\mathcal{C}^\infty$-smooth function 
with $\nu_0(\tilde P)=+\infty$. 
Define a $\mathcal{C}^\infty$-smooth function by setting
$$  
P(z)= \tilde P(\varphi(z))=\tilde P(|\varphi(z)|).
$$
Then 
$P(g(z))=\tilde P(\varphi\circ g(z))=\tilde P( R_\theta\circ \varphi(z))
=\tilde P( |R_\theta\circ \varphi(z)|)=\tilde P( |\varphi(z)|)=P(z)$. 
This means that $(z_1,z_2)\mapsto(z_1,g(z_2))$ belongs to $\mathrm{G_2}(M_P,0)$. 
Moreover $f_t(z_1,z_2):=(z_1, \varphi^{-1}\circ R_t\circ \varphi(z_2))$ is a mapping in 
$\mathrm{G_2}(M_P,0)$ for all $t\in \mathbb R$. 
In addition, it is easy to see that $M_P$ is biholomorphically equivalent to $M_{\tilde P}$, 
which is rotationally symmetric.
\end{proof}
\section{Proofs of Theorems~\ref{T1} and \ref{T11}} 
This section is devoted to the proofs of Theorems~\ref{T1} and \ref{T11}. 
For the sake of smooth exposition, we shall present these proofs in two subsections.  
\subsection{Proof of Theorem~\ref{T1}}
\begin{proof}[Proof of Theorem~\ref{T1}]

\noindent
{\bf (a)} 
Let $H(z_1,z_2)=h_1(z_1,z_2)\partial_{z_1}+h_2(z_1,z_2)\partial_{z_2}\in \mathfrak{aut}(M_P,0)$ 
be arbitrary and 
$\{\phi_t\}_{t\in\mathbb R}\subset \mathrm{Aut}(M_P)$ 
the one-parameter subgroup generated by $H$. Since $\phi_t$ is biholomorphic 
for every $t\in \mathbb R$, 
the set $\{\phi_t(0)\colon t\in\mathbb R\}$ is contained in $P_\infty(M_P)$. 
We remark that the connected component of $0$ in 
$P_\infty(M_P)$ is $\{(is,0)\colon s\in \mathbb R\}$. 
Therefore, we have $\phi_t(0,0)\subset \{(is,0)\colon s\in\mathbb R\}$. 
Consequently, we obtain $\mathrm{Re}~h_1(0,0)=0$ and $h_2(0,0)=0$. 
Hence, the holomorphic vector field $H-i\beta \partial_{z_1}$, 
where $\beta:=\mathrm{Im}~h_1(0,0)$, 
belongs to $\mathfrak{aut}_0(M_P,0$, which ends the proof.    

\noindent
{\bf (b)} In the light of $\mathrm{(a)}$, we see that 
$\mathfrak{aut}(M_P,0)=\mathfrak{g}_{-1}$, 
i.e., it is generated by $i\partial_{z_1}$. 
Denote by $\{T_t\}_{t\in\mathbb R}$ the one-parameter subgroup generated by 
$i\partial_{z_1}$, i.e., it is given by
$$
T_t(z_1,z_2)=(z_1+it,z_2), \; t \in \mathbb R. 
$$
Let $f=(f_1,f_2)\in \mathrm{Aut}\big(M_P,0\big)$ be arbitrary. 
We define $\{F_t\}_{t\in \mathbb R}$ the family of automorphisms 
by setting $F_t:=f\circ T_{-t}\circ f^{-1}$. 
Then it follows that $\{F_t\}_{t\in \mathbb R}$ is a one-parameter subgroup of 
$\mathrm{Aut}\big(M_P)$. 
Since $\mathfrak{aut}(M_P,0)=\mathfrak{g}_{-1}$, 
it follows that the holomorphic vector field generated by $\{F_t\}_{t\in \mathbb R}$ 
belongs to $\mathfrak{g}_{-1}$. 
This means that there exists a real number $\delta$ such that $F_t=T_{\delta t}$ 
for all $t\in \mathbb R$, which yields that
\begin{equation}\label{tte1}
f= T_{\delta t}\circ f\circ T_{t}, \; t \in \mathbb R .
\end{equation}
We note that if $\delta=0$, then $f=f\circ T_t$ and thus $T_t=\text{id}$ 
for any $t\in \mathbb R$, which is a contradiction. 
Hence, we may assume that $\delta\ne 0$.

We shall prove that $\delta=-1$. 
Indeed, the equation (\ref{tte1}) is equivalent to
\begin{equation*}
\begin{split}
f_1(z_1,z_2)&= f_1(z_1+it,z_2)+i\delta t, \\
f_2(z_1,z_2)&= f_2(z_1+it,z_2)
\end{split}
\end{equation*}
for all $t\in \mathbb R $. 
This implies that $\frac{\partial}{\partial z_1}f_1(z_1,z_2)=-\delta$ and 
$\frac{\partial}{\partial z_1}f_2(z_1,z_2)=0$. 
Thus, the holomorphic functions $f_1$ and $f_2$ can be re-written as follows:
\begin{equation}\label{tte3}
\begin{split}
f_1(z_1,z_2)&= -\delta z_1 +g_1(z_2),  \\
f_2(z_1,z_2)&= g_2(z_2),
\end{split}
\end{equation}
where $g_1,g_2$ are holomorphic functions on a neighborhood of $z_2=0$.

Since $M_P$ is invariant under $f$, one has
\begin{equation}\label{tte4}
\begin{split}
&\mathrm{Re}~f_1\big(it-P(z_2),z_2\big)+P\Big(f_2\big(it-P(z_2),z_2\big)\Big)=0
\end{split}
\end{equation}
for all $(z_2,t)\in \Delta_{\epsilon_0}\times (-\delta_0,\delta_0)$ for some 
$\epsilon_0,\delta_0>0$.

It follows from (\ref{tte4}) with $t=0$ and (\ref{tte3}) that
\begin{equation*}
\begin{split}
&\delta P(z_2)+\mathrm{Re}~g_1(z_2)+P\Big(g_2(z_2)\Big)=0
\end{split}
\end{equation*}
for all $z_2\in \Delta_{\epsilon_0}$. 
Since $\nu_0(P)=+\infty$, we have $\nu_0(g_1)=+\infty$, and hence $g_1\equiv 0$. 
This tells us that
\begin{equation*}
\begin{split}
P\Big(g_2(z_2)\Big)=-\delta P(z_2)
\end{split}
\end{equation*}
for all $z_2\in \Delta_{\epsilon_0}$. 
Therefore, Lemmas \ref{lemma2} and \ref{lemma1} tell us that $|g'(0)|=1$ and 
$\delta=1$. Hence, $f\in \mathrm{G_2}(M_P,0)$, which finishes the proof. 
\end{proof}

We note that if $P$ vanishes to infinite order at only the origin, 
then we have the following corollary.
\begin{corollary}\label{co1}
let $(M_P, 0)$ be as in Theorem~\ref{T1}. Assume that 
\begin{itemize}
\item[(i)] $P(z_2) \not \equiv 0$ on a neighborhood of $z_2= 0$, and
\item[(ii)] $S_\infty(P)=\{0\}$.
\end{itemize}
If $\mathfrak{aut}_0(M_P,0)$ is trivial, then
$$
\mathrm{Aut}(M_P)=\mathrm{G_2}(M_P,0)\oplus \mathrm{T^1}(M_P,0),
$$
where $\mathrm{T^1}(M_P,0)$ denotes the set of all translations $T^1_t$, 
$t\in \mathbb R$, 
defined by $T^1_t(z_1,z_2)=(z_1+it,z_2)$.
\end{corollary}

\begin{proof}
Let $f\in \mathrm{Aut}(M_P)$ be arbitrary. 
Since the origin is of infinite type, so is $f(0,0)$.  
Because of the assumption $\mathrm{(ii)}$, 
we have $P_\infty(M_P)=\{(it,0)\colon t\in \mathbb R\}$. 
This tells us that $f(0,0)=(it_0,0)$ for some $t_0\in \mathbb R$. 
Then $T^1_{-t_0}\circ f\in \mathrm{Aut}(M_P,0)$. 
Thus, the proof easily follows from Theorem \ref{T1}.
\end{proof}

In the case that $P$ is positive on a punctured disk $\Delta_{\epsilon_0}^*$, 
$\mathfrak{aut}_0(M_P,0)$ is at most one-dimensional (see \cite{NCM}). 
Moreover, if $P$ is rotational, i.e. $P(z_2)\equiv P(|z_2|)$, 
then in \cite{Ninh2} we proved that $\mathrm{Aut}(M_P,0)=\mathrm{G_2}
(M_P,0)=\{(z_1,z_2)\mapsto(z_1,e^{it}z_2)\colon t\in \mathbb R\}$. 
Therefore, we only consider the case that $P$ is not \emph{rotationally symmetricable}, 
i.e., there is no conformal map $\varphi$ with $\varphi(0)=0$ 
such that $P\circ \varphi(z_2)\equiv P\circ \varphi(|z_2|)$, 
in which case we showed that 
$\mathfrak{aut}_0(M_P,0)=\{0\}$ provided that the connected component of $0$ 
in the zero set of $P$ is $\{0\}$ (cf. Theorem \ref{T3}). 
In addition, this assertion still holds if $P$, defined on a neighborhood $U$ of $0$ 
in $\mathbb C$, satisfies the condition $\mathrm{(I)}$ (cf. \cite{Ninh1}), that is,   
\begin{itemize}
\item[(I.1)]$ \displaystyle \limsup_{\tilde U \ni z\to 0} 
|\mathrm{Re}(b z^k \frac{P'(z)}{P(z)})| =+\infty$;
\item[(I.2)] $ \displaystyle \limsup_{\tilde U \ni z\to 0} | \frac{P'(z)}{P(z)}| =+\infty$
\end{itemize}
for all $k=1,2,\ldots$ and for all $b\in \mathbb C^*$, 
where $\tilde U :=\{z\in U: P(z)\ne 0\}$. 
Therefore, as an application of Theorem \ref{T1} we obtain the following corollaries.  
\begin{corollary}\label{co2}
Let $(M_P,0)$ be as in Theorem~\ref{T1}. Assume that 
\begin{itemize}
\item[(i)] $P$ is not rotationally symmetricable,
\item[(ii)] The connected component of $0$ in the zero set of $P$ is $\{0\}$, and 
\item[(iii)] The connected component of $0$ in $S_\infty(P)$ is $\{0\}$,
\end{itemize}
then 
$$
\mathrm{Aut}(M_P,0)=G_2(M_P,0).
$$
\end{corollary}
\begin{corollary}\label{co3}
Let $(M_P,0)$ be as in Theorem~\ref{T1}. Assume that 
\begin{itemize}
\item[(i)] $P(z_2)\not \equiv 0$ on a neighborhood of $z_2= 0$,
\item[(ii)] $P$ satisfies the condition $\mathrm{(I)}$, and 
\item[(iii)] The connected component of $0$ in $S_\infty(P)$ is $\{0\}$,
\end{itemize}
then 
$$
\mathrm{Aut}(M_P,0)=G_2(M_P,0).
$$
\end{corollary}

\subsection{Proof of Theorem \ref{T11}}

\noindent 
{\bf (a)}  As a consequence of Theorem \ref{T111} in Appendix A.2, 
we see that $\mathfrak{aut}_0(M_P,0)=0$. 
Therefore, we shall prove that 
$\mathfrak{aut}(M_P,0)=\mathfrak{g}_{-1}\oplus \mathfrak{g}_0 $.
Indeed, let 
$H(z_1,z_2)=h_1(z_1,z_2)\partial_{z_1}+h_2(z_1,z_2)\partial_{z_2}\in \mathfrak{aut}(M_P,0)$ 
be arbitrary and 
$\{\phi_t\}_{t\in\mathbb R}\subset \mathrm{Aut}(M_P)$ be 
the one-parameter subgroup generated by $H$. 
Since $\phi_t$ is biholomorphic for every $t\in \mathbb R$, 
the set $\{\phi_t(0)\colon t\in\mathbb R\}$ is contained in $P_\infty(M_P)$. 
We remark that the connected component of $0$ in $P_\infty(M_P)$ is 
$\{(it_1,it_2)\colon t_1,t_2\in \mathbb R\}$. 
Therefore, we have $\phi_t(0,0)\subset \{(it_1,it_2)\colon t_1,t_2\in\mathbb R\}$. 
Consequently, we obtain $\mathrm{Re}~h_1(0,0)=0$ and $\mathrm{Re}~h_2(0,0)=0$. 
Hence, the holomorphic vector field $H-i\beta_1 \partial_{z_1}-i\beta_2 \partial_{z_2}$, 
where $\beta_j:=\mathrm{Im}~h_j(0,0)$ for $j=1,2$ belongs to $\mathfrak{aut}_0(M_P,0)$, 
which ends the proof of $\mathrm{(a)}$.    

\noindent 
{\bf (b)} By $\mathrm{(a)}$, we see that 
$\mathfrak{aut}(M_P,0)=\mathfrak{g}_{-1}\oplus \mathfrak{g}_{0}$, i.e., 
it is generated by $i\partial_{z_1}$ and $i\partial_{z_2}$. 
Denote by $\{T^j_t\}_{t\in\mathbb R}$ 
the one-parameter subgroups generated by $i\partial_{z_j}$ for $j=1,2$, i.e., 
$$
T^1_t(z_1,z_2)=(z_1+it,z_2),~T^2_t(z_1,z_2)=(z_1,z_2+it), \; t \in \mathbb R. 
$$
For any $f=(f_1,f_2)\in \mathrm{Aut}(M_P, 0)$, we define families 
$\{F^j_t\}_{t\in \mathbb R}$ 
of automorphisms by setting $F^j_t:=f\circ T^j_{-t}\circ f^{-1}$ ($j=1,2$). 
Then it follows that $\{F^j_t\}_{t\in \mathbb R},~j=1,2$, are one-parameter subgroups of 
$\mathrm{Aut}\big(M_P\big)$. 
Since $\mathfrak{aut}(M_P, 0)=\mathfrak{g}_{-1}\oplus \mathfrak{g}_{0}$, 
the holomorphic vector fields $H^j$, $j=1,2$, generated by 
$\{F^j_t\}_{t\in \mathbb R}~(j=1,2)$ belong to $\mathfrak{g}_{-1}\oplus \mathfrak{g}_{0}$. 
This means that there exist real numbers $\delta^j_1,\delta^j_2, j=1,2$, such that 
$H^j=i\delta_1^j\partial_{z_1}+i\delta_2^j\partial_{z_2}$ for $j=1,2$, which yield that
$$
F^j_t(z_1,z_2)=\big(z_1+i\delta^j_1 t,z_2+i\delta^j_2 t\big)
=T^1_{\delta^j_1 t}\circ T^2_{\delta^j_2 t},~j=1,2,~t\in\mathbb R. 
$$
This implies that
$$
f=T^1_{\delta^j_1 t}\circ T^2_{\delta^j_2 t}\circ f\circ T^j_t,
$$
which is equivalent to
\begin{equation}\label{eq2014}
\begin{split}
f_1(z_1,z_2)&=f_1(z_1+it,z_2)+i\delta^1_1t,\\
f_2(z_1,z_2)&=f_2(z_1+it,z_2)+i\delta^1_2t,\\
f_1(z_1,z_2)&=f_1(z_1,z_2+it)+i\delta^2_1t,\\
f_2(z_1,z_2)&=f_2(z_1,z_2+it)+i\delta^2_2t.
\end{split}
\end{equation}
It follows from (\ref{eq2014}) that
\begin{equation*}
\begin{split}
\frac{\partial}{\partial z_1}f_1(z_1,z_2)&=-\delta^1_1,\\
\frac{\partial}{\partial z_1}f_2(z_1,z_2)&=-\delta^1_2,\\
\frac{\partial}{\partial z_2}f_1(z_1,z_2)&=-\delta^2_1,\\
\frac{\partial}{\partial z_2}f_2(z_1,z_2)&=-\delta^2_2,
\end{split}
\end{equation*}
which tells us that 
$$
f(z_1,z_2)=(-\delta^1_1z_1-\delta^2_1 z_2,-\delta^1_2 z_1-\delta^2_2 z_2).
$$
Since $M_{P}$ is invariant under $f$, one has
\begin{equation}\label{ttte4}
\begin{split}
&\mathrm{Re}f_1\big(it-P(z_2),z_2\big)+P\Big(f_2\big(it-P(z_2),z_2\big)\Big)\\
&=\mathrm{Re}\Big(-\delta^1_1 \big(it-P(z_2)\big)-\delta^2_1 z_2\Big)
+P\Big(-\delta^1_2 \big(it-P(z_2)\big)-\delta^2_2 z_2\Big) \\
&=\delta^1_1 P(z_2)-\delta^2_1 \mathrm{Re}(z_2)
+P\Big(\delta^1_2P(z_2)-\delta^2_2 z_2\Big)=0
\end{split}
\end{equation}
for all $(z_2,t)\in \Delta_{\epsilon_0}\times (-\delta_0,\delta_0)$ 
for some $\epsilon_0,\delta_0>0$ small enough.

Since $\nu_0(P)=+\infty$, we have $\delta^2_1=0$. 
Therefore, putting $z_2=t\in (-\epsilon_0,\epsilon_0)$ in (\ref{ttte4}), we obtain the following equation
\begin{equation}\label{etq1}
P\Big(-\delta^2_2 t+\delta^1_2P(t)\Big)=-\delta^1_1 P(t)
\end{equation}
for all $t\in (-\epsilon_0,\epsilon_0)$. 
By the mean value theorem, for each $t\in (-\epsilon_0,\epsilon_0)$ 
there exists a number $\gamma(t)\in [0,1]$ such that
\begin{equation}\label{etq2}
P\Big(-\delta^2_2 t+\delta^1_2P(t)\Big)=P(-\delta^2_2 t)
+P'\Big(-\delta^2_2 t+\gamma(t)\delta^1_2P(t)\Big) \delta^1_2P(t).
\end{equation}
Because of the fact that the function $P'\Big(-\delta^2_2 t+\gamma(t)\delta^1_2P(t)\Big)$ 
vanishes to infinite order at $t=0$, by (\ref{etq1}) and (\ref{etq2}), one has
$$
P\big(-\delta^2_2 t\big)=\big(-\delta^1_1+o(1)\big)P(t), \; t \in (-\epsilon_0,\epsilon_0).
$$
Then it follows from the proof of Lemma \ref{lemma2} that it is not hard to see that 
$-\delta^1_1=-\delta^2_2=1$.

Now the equation (\ref{etq1}) becomes
$$
P\Big( t+\delta^1_2P(t)\Big)=P(t)
$$
for all $t\in (-\epsilon_0,\epsilon_0)$. By Lemma \ref{lemma-3}, this equation implies that $\delta^1_2=0$. Therefore, we conclude that $f=\mathrm{id}$, which finishes the proof of $\mathrm{(b)}$.

\noindent 
{\bf (c)} Denote by $T^1_t$ and $T^2_t$ the shifts to imaginary directions 
of the first and second components 
$$
T^1_t(z_1,z_2)=(z_1+it,z_2),~T^2_t(z_1,z_2)=(z_1,z_2+it), \; t \in \mathbb R. 
$$
Now let $f\in \mathrm{Aut}(M_P)$ be arbitrary. 
Then $f(0,0)$ is of infinite type. 
It follows from  $S_\infty(\tilde P)=\{0\}$ that we have 
$P_\infty(M_P)=\{(it,is)\colon t,s\in \mathbb R\}$.
Therefore we get $f(0,0)=(it_0,is_0)$ for some $t_0,s_0\in \mathbb R$ 
and we obtain $T^1_{-t_0}\circ T^2_{-s_0}\circ f\in \mathrm{Aut}(M_P,0)=\{\mathrm{id}\}$ 
by $\mathrm{(b)}$. The proof of $\mathrm{(c)}$ follows.
\hfill $\Box\;$

\section{Analysis of holomorphic tangent vector fields}
In this section, we study the determination of the defining function 
from holomorphic vector fields. Assume that an infinite type hypersurface $M_P$ defined by 
$\rho(z)=\mathrm{Re}~z_1+P(z_2)$ satisfying the conditions (i) and (ii) 
posed in Theorem~\ref{T3}. Theorem~{\ref{T3}} says that if there are non trivial holomorphic vector fields vanishing at the origin tangent to $M_P$, then the hypersurface $M_P$ is rotationally symmetric. 
The typical example of rotationally symmetric hypersurface is 
\begin{equation*}
M_P=\{(z_1, z_2) \in \mathbf{C}^2:\mathrm{Re}~z_1+\exp(-\dfrac{1}{|z_2|^{\alpha}})=0\},
\end{equation*}
where $\alpha >0$, as in Example~{\ref{ex2}} in \S~6. 

To prove Theorem~{\ref{T3}}, we need some lemmas. 
\begin{lemma}\label{Al3} 
Let $P:\Delta_{\epsilon_0}\to \mathbb R$ be a $\mathcal{C}^\infty$-smooth 
function satisfying that the connected component of $z=0$ 
in the zero set of $P$ is $\{0\}$ and that $P$ vanishes to infinite order at $z=0$. 
If $a, b$ are complex numbers and if $g_0, g_1, g_2$ are 
$\mathcal{C}^\infty$-smooth functions defined on $\Delta_{\epsilon_0}$ satisfying:
\begin{itemize}
\item[(A1)] $g_0(z) = O(|z|)$, $g_1(z) = O(|z|^\ell)$, and $g_2(z) = o(|z|^m)$, and
\item[(A2)] $\mathrm{Re} \Big[\big(az^m+g_2(z)\big) P^{n+1}(z)+b z^\ell\big(1+g_0(z)\big)P_{z}(z)
+g_1(z)P(z) \Big]=0$ for every $z \in \Delta_{\epsilon_0}$
\end{itemize}
for any nonnegative integers $\ell, m$ and $n$ except for the following two cases
\begin{itemize}
\item[(E1)] $\ell=1$ and $\text{Re }b = 0$, and
\item[(E2)] $m=0$ and $\text{Re } a = 0$
\end{itemize}
then $ab=0$.
\end{lemma}

The proof of Lemma \ref{Al3} for the case that $P$ is positive on 
$\Delta_{\epsilon_0}^*$  is given in \cite[Lemma $3$]{Kim-Ninh} 
(see also \cite[Lemma $1$]{NCM}). Furthermore, Lemma \ref{Al3} 
follows easily from \cite[Lemma $3$]{Kim-Ninh} and the following lemma.
\begin{lemma}\label{newlemma} 
Let $P,g_0,g_1,g_2, a,b$ be as in Lemma \ref{Al3}. 
Suppose that $\gamma: [t_0,t_\infty)\to \Delta^*_{\epsilon_0}~(t_0\in \mathbb R)$, 
where either $t_\infty\in \mathbb R$  or $t_\infty=+\infty$, 
is a solution of the initial-value problem 
$$ 
\frac{d\gamma(t)}{dt}=b \gamma^\ell(t)\big(1+g_0(\gamma(t))\big),~\gamma(t_0)=z_0,
$$
where $z_0\in \Delta^*_{\epsilon_0}$ with 
$P(z_0)\ne 0$, such that $\lim_{t\uparrow t_\infty}\gamma(t)=0$. 
Then $P(\gamma(t))\ne 0$ for every $t\in (t_0,t_\infty)$.  
\end{lemma}
\begin{proof}
To obtain a contradiction, we suppose that $P$ has a zero on $\gamma$. 
Then since the connected component of $z=0$ in the zero set of $P$ is $\{0\}$, 
without loss of generality we may assume that there exists a 
$t_1\in (t_0,t_\infty)$ such that  $P(\gamma(t))\ne 0$ for all $t\in (t_0,t_1)$ 
and $P(\gamma(t_1))=0$. Denote $u(t):=\frac{1}{2}\log |P(\gamma(t))|$ 
for $t_0<t<t_1$. It follows from (A2) that  
\begin{equation*}
u'(t)=-P^n(\gamma(t))\Big(\text{Re} \big(a \gamma^m(t) 
+o(|\gamma(t)|^m)\big)\Big)+O(|\gamma(t)|^{\ell})
\end{equation*}
for all $t_0<t<t_1$. This means that $u'(t)$ is bounded on $(t_0,t_1)$. 
Therefore, $u(t)$ is also bounded on $(t_0,t_1)$, which contradicts the fact that 
$u(t)\to -\infty$ as $t\uparrow t_1$. Hence, our lemma is proved.
\end{proof}

Following the proof of Lemma \ref{Al3} (see also \cite[Lemma $1$]{NCM}), 
we have the following corollary.
\begin{corollary}\label{cor3} 
Let $P:\Delta_{\epsilon_0}\to \mathbb R$ be a $\mathcal{C}^\infty$-smooth 
function satisfying that the connected component of $z=0$ in the zero set of 
$P$ is $\{0\}$ and that $P$ vanishes to infinite order at $z=0$. 
If $b$ is a complex number and if $g$ is a $\mathcal{C}^\infty$-smooth 
function defined on $\Delta_{\epsilon_0}$ satisfying:
\begin{itemize}
\item[(B1)] $g(z) = O(|z|^{k+1})$, and
\item[(B2)] $\text{Re} \Big[\big(b z^k+g(z)\big)P_{z}(z)\Big]= 0$ for every 
$z \in \Delta_{\epsilon_0}$
\end{itemize}
for some nonnegative integer $k$, except the case $k= 1$ and 
$\text{Re} (b)=0$, then $b=0$.
\end{corollary}

Now we are ready to prove Theorem~{\ref{T3}}. 
\begin{proof}[Proof of Theorem~{\ref{T3}}]
The CR hypersurface germ $(M_P,0)$ at the origin in $\CC^2$ 
under consideration is defined by the equation 
$$
\rho (z_1, z_2):= \mathrm{Re}~z_1 + P(z_2) = 0,
$$
where $P$ is $\mathcal{C}^\infty$-smooth functions satisfying 
the two conditions of this theorem.   
Recall that $P$ vanishes to infinite order at $z_2=0$ in particular.

Then we consider a holomorphic vector field 
$H=h_1(z_1,z_2)\partial_{z_1}+h_2(z_1,z_2)\partial_{z_2}$ 
defined on a neighborhood of the origin. We only consider $H$ 
that is tangent to $M_P$. 
This means that they satisfy the identity
\begin{equation}\label{eq221t}
(\mathrm{Re}~ H) \rho(z)=0,\; \forall z \in M_P.
\end{equation}

Expand $h_1$ and $h_2$ into the Taylor series at the origin 
$$
h_1(z_1,z_2)=\sum\limits_{j,k=0}^\infty a_{jk} z_1^j z_2^k=\sum_{j=0}^\infty a_j(z_2)z_1^j;
h_2(z_1,z_2)=\sum\limits_{j,k=0}^\infty b_{jk} z_1^jz_2^k=\sum_{j=0}^\infty b_j(z_2)z_1^j,
$$
where $a_{jk}, b_{jk}\in \mathbb C$ and $a_j,b_j$ are holomorphic functions 
for every $j\in \mathbb N$. We note that $a_{00}=b_{00}=0$ since $h_1(0,0)=h_2(0,0)=0$.

By a simple computation, we have 
\begin{equation*}
\begin{split}
\rho_{z_1}(z_1,z_2)= \frac{1}{2},~\rho_{z_2}(z_1,z_2)= P_{z_2}(z_2),
\end{split}
\end{equation*}
and the equation (\ref{eq221t}) can thus be re-written as
\begin{equation}\label{eq222t}
\begin{split}
&\mathrm{Re} \Big[\frac{1}{2}h_1(z_1,z_2) +P_{z_2}(z_2)h_2(z_1,z_2)\Big ]=0
\end{split}
\end{equation}
for all $(z_1,z_2)\in M_P$.
Since the point $(it-P(z_2), z_2)$ is in $M_P$ with $t$ small enough, 
the above equation again admits a new form
\begin{equation}\label{eq223t}
\mathrm{Re}\Big[ \frac{1}{2}\sum_{j,k=0}^\infty a_{jk}\big(it-P(z_2)\big)^j z_2^k
+P_{z_2}(z_2) \sum_{m,n=0}^\infty b_{mn} \big(it-P(z_2)\big)^m z_2^n\Big]=0
\end{equation}
for all $z_2\in \mathbb C$ and for all $t\in\mathbb R$ 
with $|z_2|<\epsilon_0$ and $|t|<\delta_0$, where $\epsilon_0>0$ and $\delta_0>0$ 
are small enough. 
The goal is to show that $H\equiv 0$. 
Striving for a contradiction, we suppose that $H\not\equiv 0$. 
Since $P_{z_2}(z_2)$ vanishes to infinite order at $0$, 
we notice that if $h_2\equiv 0$, then (\ref{eq222t}) shows that $h_1\equiv 0$. 
So, we must have  $h_2\not\equiv 0$.

We now divide the argument into two cases as follows.

\noindent
{\bf Case 1.} {\boldmath $h_1\not \equiv 0$.} 
In this case let us denote by $j_0$ the smallest integer such that $a_{j_0 k}\ne 0$ 
for some integer $k$. 
Then let $k_0$ be the smallest integer such that $a_{j_0 k_0}\ne 0$. 
Similarly, let $m_0$ be the smallest integer such that $b_{m_0 n}\ne 0$ for some integer $n$. 
Then denote by $n_0$ the smallest integer such that $b_{m_0 n_0}\ne 0$. 
We see that $j_0\geq 1$ if $k_0=0$, and $m_0\geq 1$ if $n_0=0$. 
Since $P(z_2)=o(|z_2|^j)$ for any $j\in \mathbb N$, inserting $t=\alpha P(z_2)$ 
into (\ref{eq223t}), where $\alpha\in \mathbb R$ will be chosen later, one has
\begin{equation}\label{eq224t}
\begin{split}
\mathrm{Re} \Big[&\frac{1}{2} a_{j_0k_0}(i\alpha -1)^{j_0}(P(z_2))^{j_0}
\big(z_2^{k_0}+o(|z_2|^{k_0})\big) \\
&+ b_{m_0n_0}(i\alpha -1)^{m_0}\big(z_2^{n_0}+o(|z_2|^{n_0})\big) 
(P(z_2))^{m_0}P_{z_2}(z_2)  \Big ]=0
\end{split}
\end{equation}
for all $z_2\in \Delta_{\epsilon_0}$. 
We note that in the case $k_0=0$ and $\mathrm{Re}(a_{j_0 0})=0$, 
$\alpha$ is chosen in such a way that $\mathrm{Re}\big( (i\alpha-1)^{j_0}a_{j_0 0}\big)\ne 0$. 
Then (\ref{eq224t}) yields that $j_0>m_0$ by virtue of the fact 
that $P_{z_2}(z_2)$ and $P(z_2)$ vanish to infinite order at $z_2=0$. 

We now consider two subcases as follows:
\medskip

\noindent
{\bf Subcase 1.1.} {\boldmath $m_0\geq 1$.}
If $n_0=1$, then the number $\alpha$ can also be chosen such that
$\text{Re }(b_{m_01}(i\alpha-1)^{m_0})\ne 0$. Therefore, $(\ref{eq224t})$ 
contradicts Lemma \ref{Al3}. Hence, we must have $m_0=0$.

\smallskip
\noindent
{\bf Subcase 1.2.} {\boldmath $m_0=0$.}
In addition to this condition, if $n_0>1$, or if $n_0=1$ and $\text{Re}(b_{01})\ne 0$,
then (\ref{eq224t}) contradicts Lemma \ref{Al3}. Therefore, we may assume that $n_0=1$ 
and $\mathrm{Re}(b_{01})=0$. By a change of variable in $z_2$ 
as in \cite[Lemma $1$]{Kim-Ninh}, we may assume that $b_0(z_2)\equiv i z_2$. 

Next, we shall prove that $b_m\equiv 0$ for every $m \in \mathbb N^*$. 
Indeed, suppose otherwise. Then let $m_1>0$ be the smallest integer 
such that $b_{m_1}\not \equiv  0$. Thus it can be written as follows:
$$ 
b_{m_1}(z_2)=b_{m_1 n_1} z_2^{n_1}+o(z_2^{n_1})
$$
where $n_1=\nu_0(b_{m_1})$ and $b_{m_1 n_1}\in \mathbb C^*$. 
Take a derivative by $t$ at $t=\alpha P(z_2)$ of both sides 
of the equation (\ref{eq223t}) and notice that $\nu_0(P)=+\infty$. 
One obtains that
\begin{equation}\label{eq?6}
\begin{split}
& \mathrm{Re}\Big[im_1\big(\alpha i-1\big)^{m_1-1}\big(P(z_2)\big)^{m_1-1}
\big(b_{m_1 n_1} z_2^{n_1}+o(|z_2|^{n_1})\big) P_{z_2}(z_2)\\
&+ j_1\big( a_{j_1 k_1}z_2^{k_1}+o(|z_2|^{k_1})\big) 
\big(\alpha i-1\big)^{j_1-1}\big(P(z_2)\big)^{j_1-1}\Big]=0
\end{split}
\end{equation}
for all $z_2\in  \Delta_{\epsilon_0}$, where $j_1,n_1\in\mathbb N$ and 
$a_{j_1 k_1}\in \mathbb C$. 

Following the argument as above, by Lemma \ref{Al3} and Corollary \ref{cor3},  
we conclude that
$m_1=n_1=1$ and $b_{1}(z_2)\equiv -\beta_1 z_2\big(1+O(z_2)\big)$ for some 
$\beta_1\in \mathbb R^*$. We claim that $b_1(z_2)\equiv -\beta_1 z_2$. 
Otherwise, the equation (\ref{eq?6}) implies that
\begin{equation}\label{eq?66}
\mathrm{Re}\big(iz_2P_{z_2}(z_2)\big)\equiv \mathrm{Re}\Big[a z^\ell\big(1+O(|z_2|)\big)P_{z_2}(z_2)\Big]+O(P(z_2)) 
\end{equation}
on $\Delta_{\epsilon_0}$ for some $a\in \mathbb C^*$ and $\ell \geq 2$. On the other hand, 
since $\nu_0(P)=+\infty$, inserting $t=0$ into (\ref{eq223t}) one has 
\begin{equation}\label{eq?666}
\mathrm{Re}\Big[ iz_2\Big( 1 -i\beta_1 \big(1+O(|z_2|)\big)P(z_2)\Big)P_{z_2}(z_2)+\big(a_{10}+o(1)\big) P(z_2)\Big]\equiv 0
\end{equation}
on $\Delta_{\epsilon_0}$. 
Therefore, subtracting (\ref{eq?66}) from (\ref{eq?666}) yields
\begin{equation}\label{neq?6}
\mathrm{Re}\Big[ i a z^{\ell}_2\big(1+O(|z_2|)\big) P_{z_2}(z_2)+\big(a_{10}+o(1)\big) P(z_2)\Big]\equiv 0
\end{equation}
on $\Delta_{\epsilon_0}$, which is impossible by Lemma \ref{Al3}. 
Hence, $b_1(z_2)\equiv -\beta_1 z_2$. 

Using the same argument as above, we obtain that 
$b_m(z_2)=\beta_m i^{m+1}z_2$ for every $m\in\mathbb N^*$, 
where $\beta_m\in \mathbb R^*$ for every $m\in\mathbb N^*$. 

Putting $t=\alpha P(z_2)$ in (\ref{eq223t}), one has 
\begin{equation}\label{eqt1}
\begin{split}
\mathrm{Re}\Big[& iz_2\Big( 1 +i\beta_1(i\alpha -1) P(z_2)
+\cdots +i^{m} \beta_m(i\alpha-1)^m P^m(z_2)+\cdots \Big)P_{z_2}(z_2)\\
&+\big(a_{10}+o(1)\big)P(z_2)\Big]\equiv 0
\end{split}
\end{equation}
on $\Delta_{\epsilon_0}$. 
On the other hand, taking a derivative both sides of (\ref{eq223t}) by $t$ 
at $t=\alpha P(z_2)$ , 
one also has
\begin{equation*}\label{eqt2}
\begin{split}
\mathrm{Re}\Big[& iz_2\Big( i^2\beta_1 +i^3 2\beta_2(i\alpha -1) P(z_2)
+\cdots +i^{m+2} m\beta_m(i\alpha -1)^{m-1} P^m(z_2)+\cdots \Big)P_{z_2}(z_2)\\
&+\frac{1}{2} \sum_{j=1}^\infty\sum_{k=0}^\infty j a_{jk}
\big(i\alpha-1\big)^{j-1}P^{j-1}(z_2) z_2^k \Big]\equiv 0,
\end{split}
\end{equation*}
or equivalently 
\begin{equation}\label{eqt2}
\begin{split}
\mathrm{Re}\Big[ &iz_2\Big( 1 +i 2\frac{\beta_2}{\beta_1}(i\alpha -1) 
P(z_2)+\cdots +i^{m} m\frac{\beta_m}{\beta_1}(\alpha i-1)^{m-1} 
P^m(z_2)+\cdots \Big)P_{z_2}(z_2)\\
&-\frac{1}{2\beta_1} \sum_{j=1}^\infty\sum_{k=0}^\infty j a_{jk}
\big(i\alpha-1\big)^{j-1}P^{j-1}(z_2) z_2^k \Big]\equiv 0
\end{split}
\end{equation}
on $\Delta_{\epsilon_0}$.

Now it follows from (\ref{eqt1}) and (\ref{eqt2}) that 
$$
2\beta_2/\beta_1=\beta_1,3\beta_3/\beta_1=\beta_2,\ldots,
m\beta_m/\beta_1=\beta_{m-1},\ldots,
$$
for otherwise, subtracting (\ref{eqt1}) from (\ref{eqt2}) 
one gets an equation depending on $\alpha$ which contradicts Lemma \ref{Al3} 
for some $\alpha\in \mathbb R$. 
Therefore, $\beta_m=\frac{(\beta_1)^m}{m!}$ for all $m\in \mathbb N^*$, and hence
$$
h_2(z_1,z_2)=iz_2 \Big(1+i\beta_1 z_1+i^2\frac{\beta_1^2}{2!}z_1^2+\cdots+i^m
\frac{\beta_1^m}{m!}z_1^m+\ldots\Big)=iz_2e^{i\beta_1 z_1}
$$
for all $z_2\in \Delta_{\epsilon_0}$. Moreover, the equation (\ref{eq223t}) becomes
\begin{equation}\label{eqt223}
\mathrm{Re}\Big[ \frac{1}{2}\sum_{j,k=0}^\infty a_{jk}\big(it-P(z_2)\big)^j z_2^k
+iz_2 P_{z_2}(z_2)\exp\Big(i\beta_1\big(it-P(z_2)\big)\Big) \Big]=0
\end{equation}
for all $(z_2,t)\in \Delta_{\epsilon_0}\times (-\delta_0,\delta_0)$.

Denote $f(z_2,t):=\mathrm{Re}\Big[\sum_{j,k=0}^\infty a_{jk}\big(it-P(z_2)\big)^j z_2^k\Big]$ 
for $(z_2,t)\in \Delta_{\epsilon_0}\times (-\delta_0,\delta_0)$. Then (\ref{eqt223}) 
tells us that
$$
f(z_2,t)=-2\mathrm{Re}\Big[iz_2 P_{z_2}(z_2)\exp\Big(i\beta_1\big(it-P(z_2)\big)\Big)
 \Big],~\forall~(z_2,t)\in \Delta_{\epsilon_0}\times (-\delta_0,\delta_0).  
$$
This implies that  $f(z_2,t)$ vanishes to infinite order at $z_2=0$ for every $t$ 
since $P_{z_2}(z_2)$ vanishes to infinite order at $z_2=0$ and $f_t(z_2,t)=-\beta_1 f(z_2,t)$. 
Consequently, one must have $a_{jk}=0$ for every $k\in \mathbb N^*$ and 
$j\in \mathbb N$, and thus
$$
f(z_2,t)=\mathrm{Re}\Big[ \sum_{j=0}^\infty a_{j0}\big(it-P(z_2)\big)^j\Big].
$$ 

Furthermore, the equation $f_t(z_2,0)=-\beta_1 f(z_2,0)$ yields
$$ 
\mathrm{Re}(ia_{10})+2\mathrm{Re}(ia_{20})(-P(z_2))+o(P(z_2))=-\beta_1
\Big(\mathrm{Re}(a_{10}) (-P(z_2))+o(P(z_2))\Big),
$$
which implies that $\mathrm{Re}(ia_{10})=0, 2\mathrm{Re}(ia_{20})=-\beta_1 \mathrm{Re}
(a_{10})=-\beta_1a_{10}$. Similarly, it follows from the equation $f_{tt}(z_2,0)=-\beta_1
 f_t(z_2,0)=\beta_1^2 f(z_2,0)$ that 
$$
2\mathrm{Re}(i^2a_{20})+3!\mathrm{Re}(i^2a_{30})(-P(z_2))+o(P(z_2))=\beta_1^2
\Big(\mathrm{Re}(a_{10}) (-P(z_2))+o(P(z_2))\Big),
$$
which again implies that 
$\mathrm{Re}(i^2a_{20})=0, 3!\mathrm{Re}(i^2a_{30})=\beta^2_1 \mathrm{Re}(a_{10})=\beta_1^2
a_{10}$. Continuing this process, we conclude that 
$a_{m0}=\frac{(i\beta_1)^{m-1}}{m!}a_{10}$ for every $m\in\mathbb N^*$, and hence 
$$
h_1(z_1,z_2)\equiv a_{10}\frac{e^{i\beta_1 z_1}-1}{i\beta_1}.
$$  
This implies $a_{10}\ne 0$ as $h_1$ does not vanish identically.  

Without loss of generality, we may assume that 
$a_{10}<0$. The case that $a_{10}>0$ will follow by a similar argument. 

Now the equation (\ref{eqt223}) with $t=0$ is equivalent to
\begin{equation}\label{eqtt223}
2\mathrm{Re}\Big[iz_2 P_{z_2}(z_2)\exp\Big(-i\beta_1 P(z_2)\Big) \Big]
=a_{10}\frac{\sin(\beta_1 P(z_2))}{\beta_1} 
\end{equation}
for all $z_2\in \Delta_{\epsilon_0}$.

Since $P$ is continuous at $z_2=0$, we may assume that $|P(z_2)|<\frac{\pi}{|\beta_1|}$ 
for every $|z_2|<\epsilon_0$. 
Moreover, because of the property (i) of $P$ 
there exists a real number $r\in (0,\epsilon_0)$ such that 
$0<|P(r)|<\frac{\pi}{|\beta_1|}$ and $re^{\pi/|a_{10}|}<\epsilon_0$. 

Fix $r$ and let $\gamma: (-\infty,+\infty)\to \Delta^*_{\epsilon_0}$ be a flow of 
the following equation
$$
\frac{d\gamma(t)}{dt}=i \gamma(t)\exp\Big(-i\beta_1 P(\gamma(t))\Big) ,~\gamma(0)=r.
$$
Denote $u(t):=P(\gamma(t))$ for $-\infty <t<+\infty$. Then (\ref{eqtt223}) is equivalent to
$$ 
u'(t)=a_{10} \frac{\sin(\beta_1 u)}{\beta_1}.
$$
A short computation shows that this differential equation has the solution
\begin{equation}\label{tn1}
\begin{split}
P(\gamma(t))=u(t)=\frac{2}{\beta_1}\arctan\Big\{\tan(\beta_1 
P(r)/2)e^{a_{10}t}\Big\}, ~-\infty <t<+\infty.
\end{split}
\end{equation}
Therefore, we have for $-\infty <t<+\infty$
\begin{equation*}
\begin{split}
\gamma(t)&=r\exp\Big[ \int_0^t ie^{-i\beta_1 P(\gamma(s))}ds\Big]\\
 &=r\exp\Big[ \int_0^t i\exp\Big(-2i\arctan\Big\{\tan(\beta_1
 P(r)/2)e^{a_{10}s}\Big\}\Big)ds\Big], 
\end{split}
\end{equation*}
and thus
\begin{equation*}
\begin{split}
|\gamma(t)|=r\exp\Big[ \int_0^t \sin\Big(2\arctan\Big\{\tan(\beta_1 
P(r)/2)e^{a_{10}s}\Big\}\Big)ds\Big].
\end{split}
\end{equation*}

By employing some trigonometric identities, we obtain the following
\begin{equation*}
\begin{split}
r^+&:= \lim_{t\to -\infty}|\gamma(t)|=r\exp\Big[ \int_0^{-\infty} 
\sin\Big(2\arctan\Big\{\tan(\beta_1 P(r)/2)e^{a_{10}s}\Big\}\Big)ds\Big]\\
&=r\exp\Big[ \int_0^{-\infty} \sin\Big( \pi -2\arctan\Big\{\frac{e^{-a_{10}s}}{\tan(\beta_1 
P(r)/2)}\Big\}\Big)ds\Big]\\
&=r\exp\Big[ \int_0^{-\infty} \sin\Big( 2\arctan\Big\{\frac{e^{-a_{10}s}}{\tan(\beta_1 
P(r)/2)}\Big\}\Big)ds\Big]\\
&=r\exp\Big[- \int_0^{+\infty} \sin\Big( 2\arctan\Big\{\frac{e^{a_{10}s}}{\tan(\beta_1 
P(r)/2)}\Big\}\Big)ds\Big]\\
&=r\exp\Big[-2 \int_0^{+\infty} \frac{\frac{e^{a_{10}s}}{\tan(\beta_1 P(r)/2)}}
{1+\Big(\frac{e^{a_{10}s}}{\tan(\beta_1 P(r)/2)}\Big)^2}ds\Big]\\
&=r\exp\Big[- \frac{2}{a_{10}}\int_0^{+\infty} \frac{d\Big(\frac{e^{a_{10}s}}{\tan(\beta_1 
P(r)/2)}\Big)}{1+\Big(\frac{e^{a_{10}s}}{\tan(\beta_1 P(r)/2)}\Big)^2}\Big]\\
&=r\exp\Big[\frac{2}{a_{10}}\arctan\Big(\frac{1}{\tan(\beta_1 P(r)/2)}\Big) \Big]\\
&\leq r\exp(\frac{\pi}{|a_{10}|})<\epsilon_0.
\end{split}
\end{equation*}
Therefore, there exists a sequence $\{t_n\}\subset \mathbb R$ such that $t_n\to -\infty$ and $\gamma(t_n)\to r^+e^{i\theta_0}$  as $n\to \infty$ for some $\theta_0\in [0,2\pi)$. Moreover, $|P(r^+e^{i\theta_0})|<|\frac{\pi}{\beta_1}|$. However, since $a_{10}<0$ and since $P$ is continuous on $\Delta_{\epsilon_0}$, it follows from (\ref{tn1}) that 
\begin{equation*}
|P(r^+e^{i\theta_0})|=|P(\lim_{n\to \infty}\gamma(t_n))|=|\lim_{n\to \infty} P(\gamma(t_n))|=|\frac{\pi}
{\beta_1}|, 
\end{equation*}
which is impossible.
 
Therefore, altogether we must have $h_1\equiv 0$.

\noindent
{\bf Case 2.} {\boldmath $h_1\equiv 0$.} 

We shall follow the proof of  \cite[Lemma $12$]{Ninh1}. 
In this case, (\ref{eq223t}) is equivalent to 
\begin{equation}\label{eq?5}
\mathrm{Re}\Big[P_{z_2}(z_2)\sum_{m=0}^\infty \big(it-P(z_2)\big)^m b_m(z_2)\Big]=0
\end{equation}
for all $(z_2,t)\in \Delta_{\epsilon_0}\times(-\delta_0,\delta_0)$, where $\epsilon_0>0$ 
and $\delta_0>0$ are small enough.

Since $h_2\not \equiv 0$, there is the smallest $m_0$ such that $b_{m_0}\not\equiv 0$ 
and thus it can be written as follows:
$$ 
b_{m_0}(z_2)=b_{m_0 n_0} z_2^{n_0}+o(z_2^{n_0}),
$$
where $n_0=\nu_0(b_{n_0})$ and $b_{m_0 n_0}\in \mathbb C^*$. 
Moreover, since $P(z_2)=o(|z_2|^{n_0})$ it follows from (\ref{eq?5}) with $t=\alpha P(z_2)$ 
($\alpha\in \mathbb R$ will be chosen later) that
$$
\mathrm{Re}\Big[\big(i\alpha-1 \big)^{m_0} \big(b_{m_0 n_0}z_2^{n_0}+o(|z_2|^{n_0})\big)
P_{z_2}(z_2)\Big]=0
$$
for every $z_2\in\Delta^*_{\epsilon_0}$. 
Notice that if $m_0>0$, then we can choose 
$\alpha$ so that 
\begin{equation*}
\mathrm{Re}\Big[ b_{m_0 n_0}\big(i\alpha-1 \big)^{m_0}\Big]\ne 0.
\end{equation*} 
Therefore, it follows from Corollary \ref{cor3} that $m_0=0, n_0=1$, 
and $\mathrm{Re}(b_{m_0 n_0})=\mathrm{Re}(b_{01})=0$. By a change of variable in $z_2$ (cf. \cite[Lemma $1$]
{Ninh1}), we can assume that $b_0(z_2)\equiv iz_2$. 

Next, we shall prove that $b_m\equiv 0$ for every $m\in \mathbb N^*$. Indeed, suppose otherwise. Then using the same argument as in the Subcace $1.1$, 
 $b_m(z_2)\equiv i^{m+1}\frac{(\beta_1)^m}{m!} z_2$ for every $m\in \mathbb N^*$. 
Therefore, $h_2(z_2)\equiv iz_2e^{i\beta_1 z_1}$.

Now the equation (\ref{eq?5}) with $t=0$ is equivalent to

\begin{equation}\label{neqtt223}
2\mathrm{Re}\Big[iz_2 P_{z_2}(z_2)\exp\Big(-i\beta_1 P(z_2)\Big) \Big]=0
\end{equation}
for all $z_2\in \Delta_{\epsilon_0}$.

Let $\gamma: (-\infty,+\infty)\to \Delta^*_{\epsilon_0}$ be a flow of the following equation
$$
 \frac{d\gamma(t)}{dt}=i \gamma(t)\exp\Big(-i\beta_1 P(\gamma(t))\Big) ,~\gamma(0)=r,
$$
where $0<r<\epsilon_0$ with $P(r)\ne 0$. 
Denote $u(t):=P(\gamma(t))$ for $-\infty <t<+\infty$. Then (\ref{neqtt223}) is equivalent to
$$ 
u'(t)=0,~-\infty <t<+\infty. 
$$
This tells us that $u(t)\equiv u(0)$, and therefore $P(\gamma(t))=P(r)$ for all $t\in \mathbb R$. 
Hence, we have
$$
\gamma(t)=r\exp\Big( ie^{-i\beta_1 P(r)}t\Big) 
$$
for all $t\in \mathbb R$, and thus
\begin{equation}\label{tq1}
|\gamma(t)|=r\exp\Big(\sin\big(\beta_1 P(r)\big) t\Big).
\end{equation}
Without loss of generality, we may assume that $\beta_1P(r)<0$. 
Then (\ref{tq1}) implies that $\gamma(t)\to 0$ as $t\to +\infty$, 
hence 
\begin{equation*}
P(r)=P(\gamma(t))=\lim_{t\to +\infty}P(\gamma(t))=P(0)=0.
\end{equation*} 
This is a contradiction. Therefore, $h_2(z_2)\equiv iz_2$.

Consequently, the equation (\ref{eq?5}) is now equivalent to
\begin{equation*}
\mathrm{Re}\Big[iz_2P'(z_2)\Big]=0
\end{equation*}
for all $z_2\in \Delta_{\epsilon_0}$, and thus it follows from 
\cite[Lemma 4]{Kim-Ninh} that $P$ is rotational. This ends the proof.
\end{proof}

\section{Examples}
\begin{example}
For $\alpha,C>0$, let $P$ be a function given by
\begin{equation*}
\begin{split} 
P(z_2)&=
\begin{cases}
\exp\Big(-\frac{C}{|\mathrm{Re}(z_2)|^\alpha}\Big) ~&\text{if}~ \mathrm{Re}(z_2)\ne 0, \\
0 &\text{if}~\mathrm{Re}(z_2)=0.
\end{cases}    
\end{split}
\end{equation*}
We note that the function $P$ satisfies the condition 
$\mathrm{(I)}$ (see \cite[Example $1$]{Ninh1}).
Moreover, since the function $\tilde P$, 
defined by $\tilde P(z_2)=\exp\Big(-\frac{C}{|z_2|^\alpha}\Big)$ 
if $z_2\ne 0$ and $\tilde{P}(0)=0$,
vanishes to infinite order only at the origin, it follows from 
Theorem~\ref{T11} that $\mathfrak{aut}_0(M_P,0)=0$ and
\begin{equation*}
\begin{split}
\mathfrak{aut}(M_P,0)=\mathfrak{g}_{-1}\oplus \mathfrak{g}_{0}=\{i\beta_1\partial_{z_1}+i\beta_2 \partial_{z_2}\colon \beta_1,\beta_2\in \mathbb R\}.
\end{split}
\end{equation*}
In addition, one obtains that 
$\mathrm{Aut}(M_P,0)=\{\text{id}\}$ and $\mathrm{Aut}
(M_P)=\{(z_1,z_2)\mapsto (z_1+it,z_2+is)\colon t,s\in\mathbb R\}$. 
\end{example}

\begin{example}\label{ex2}
Denote by $M_P$ the following hypersurface
$$
M_P:=\{(z_1,z_2)\in\mathbb C^2\colon \mathrm{Re}~z_1+P(z_2)=0\}.
$$
Let $P_1,P_2$ be functions given by
\begin{equation*}
\begin{split} 
P_1(z_2)&=
\begin{cases}
\exp\Big(-\frac{1}{|z_2|^\alpha}\Big) ~&\text{if}~ z_2\ne 0\\
0 &\text{if}~z_2=0,
\end{cases}   \\
P_2(z_2)&=
\begin{cases}
\exp\Big(-\frac{1}{|z_2|^\alpha}+\mathrm{Re}(z_2^m)\Big) ~&\text{if}~ z_2\ne 0\\
0 &\text{if}~z_2=0
\end{cases}  
\end{split}
\end{equation*}
where $\alpha>0$ and $m\in \mathbb N^*$. 

It is easy to check that $S_\infty(P_1)=S_\infty(P_2)=\{0\}$. 
Moreover, $P_1,P_2$ are positive on $\mathbb C^*$, $P_1$ is rotational, 
and $P_2$ is not rotational. 
Therefore, by Theorems \ref{T1}, \ref{T11} and \ref{T3}, \cite[Theorem B]{Ninh2}, and Corollaries \ref{co1} and \ref{co2}, we obtain the followings: 
\begin{equation*}
\begin{split}
\mathfrak{aut}_0(M_{P_1},0)&=\{i\beta z_2\partial_{z_2}\colon \beta\in \mathbb R\},\\
 \mathfrak{aut}(M_{P_1},0)&=\mathfrak{g}_{-1}\oplus \mathfrak{aut}_0(M_{P_1},0)\\
                          &=\{i\beta_1\partial_{z_1}+i\beta_2 z_2\partial_{z_2}\colon \beta_1,\beta_2\in \mathbb R\},\\ 
\mathfrak{aut}_0(M_{P_2},0)&=0, \\
\mathfrak{aut}(M_{P_2},0)&=\mathfrak{g}_{-1}=\{i\beta \partial_{z_1}\colon \beta\in \mathbb R\},
\end{split}
\end{equation*}
and
\begin{equation*}
\begin{split}
\mathrm{Aut}(M_{P_1},0)&=\{(z_1,z_2)\mapsto(z_1,e^{it}z_2)\colon t\in\mathbb R\},\\ 
\mathrm{Aut}(M_{P_1})&=\mathrm{Aut}(M_{P_1},0)\oplus \mathrm{T^1}(M_{P_1})\\
                     &=\{(z_1,z_2)\mapsto(z_1+is,e^{it}z_2)\colon s,t\in\mathbb R\},\\
\mathrm{Aut}(M_{P_2},0)&=\{(z_1,z_2)\mapsto(z_1,e^{2k\pi i/m}z_2)\colon k=0,\ldots,m-1\},\\
\mathrm{Aut}(M_{P_2})&=\mathrm{Aut}(M_{P_2},0)\oplus \mathrm{T^1}(M_{P_2})\\
                     &=\{(z_1,z_2)\mapsto(z_1+it,e^{2k\pi i/m}z_2)\colon t\in \mathbb R, k=0,\ldots,m-1\}.
\end{split}
\end{equation*}
\end{example}
\section*{Appendix A}
\subsection*{A.1. Leau-Fatou flower theorem}
The Leau-Fatou flower theorem states that it is possible to find invariant 
simple connected domains 
containing $0$ on the boundaries such that, on each domain, 
a conformal map which tangent to the identity is conjugated to 
a parabolic automorphism of the domain 
and each point in the domain is either attracted to or repelled from $0$. 
For more details we refer the reader to \cite{Abate, Br}.
These domains are called petals and their existence is predicted 
by the Leau-Fatou flower theorem. 
To give a simple statement of such a result, 
we note that if $g(z) = z + a_r z^r + O(z^{r+1})$ with $r > 1$ 
and $a_r\ne 0$, it is possible to perform a holomorphic change of variables 
in such a way that $g$ becomes conjugated to $g(z) = z + z^r + O(z^{r+1})$. 
The number $r$ is the order of $g$ at $0$. 
With these preliminary considerations at hand we have 
\begin{theorem}[Leau-Fatou flower theorem]\label{flower} 
Let $g(z) = z + z^r + O(z^{r+1})$ with $r > 1$. 
Then there exist $2(r - 1)$ domains called petals, $P_j^\pm$, symmetric 
with respect to the $(r - 1)$ directions $\arg z = 2\pi q/(r - 1), q = 0,\ldots,r - 2$ 
such that $P_j^+ \cap P_k^+ = \emptyset$ and $P_j^- \cap P_k^- = \emptyset$ for $j\ne k, 0\in \partial P_j^\pm$, 
each petal is biholomorphic to the right-half plane $H$, 
and $g^k(z) \to 0$ as $k \to \pm\infty$ for all $z \in P_j^\pm$, where $g^{k}=(g^{-1})^{-k}$ for $k<0$. 
Moreover for all $j$, the map $g\mid_{P_j^\pm}$ is holomorphically conjugated 
to the parabolic automorphism $z\to z + i$ on $H$. 
\end{theorem}

\subsection*{A.2. Holomorphic tangent vector fields on the tubular model}
In the case that an infinite type model is tubular, we have the following theorem.
\begin{theorem}\label{T111}
Let $\tilde P$ be a $\mathcal{C}^\infty$-smooth function defined on a neighborhood of $0$ in $\mathbb C$ satisfying
\begin{itemize}
\item[(i)] $\tilde P(x)\not \equiv 0$ on a neighborhood of $x= 0$ in $\mathbb R$, and
\item[(ii)] $\tilde P$ vanishes to infinite order at $z_2=0$.
\end{itemize}
Denote by $P$ a $\mathcal{C}^\infty$-smooth function defined by setting 
$P(z_2):=\tilde P(\mathrm{Re}~z_2)$. Then $\mathfrak{aut}_0(M_P,0)=0$.
\end{theorem}
\begin{proof}
Suppose that $H=h_1(z_1,z_2)\partial_{z_1}+h_2(z_1,z_2)\partial_{z_2}$ 
is a holomorphic vector field defined on 
a neighborhood of the origin satisfying $H(0)=0$. We only consider $H$ 
that is tangent to $M_P$, which means that it satisfies the identity
\begin{equation}\label{eq221}
(\mathrm{Re}~ H) \rho(z)=0, \; z \in M_P.
\end{equation}

Expand $h_1$ and $h_2$ into the Taylor series at the origin 
$$
h_1(z_1,z_2)=\sum\limits_{j,k=0}^\infty a_{jk} z_1^j z_2^k, \; 
h_2(z_1,z_2)=\sum\limits_{j,k=0}^\infty b_{jk} z_1^jz_2^k,
$$
where $a_{jk}, b_{jk}\in \mathbb C$. We note that $a_{00}=b_{00}=0$ 
since $h_1(0,0)=h_2(0,0)=0$.

By a simple computation, we have 
\begin{equation*}
\begin{split}
\rho_{z_1}(z_1,z_2)= \frac{1}{2},~\rho_{z_2}(z_1,z_2)= P_{z_2}(z_2)=\frac{1}{2}P'(x),
\end{split}
\end{equation*}
where $x=\mathrm{Re}(z_2)$, and the equation (\ref{eq221}) 
can thus be re-written as
\begin{equation}\label{eq222}
\begin{split}
&\mathrm{Re} \Big[\frac{1}{2}h_1(z_1,z_2) +P_{z_2}(z_2)h_2(z_1,z_2)\Big ]=0
\end{split}
\end{equation}
for all $(z_1,z_2)\in M_P$.
Since the point $(it-P(z_2), z_2)$ is in $M_P$ with $t$ small enough, 
the above equation again admits a new form
\begin{equation}\label{eq223}
\mathrm{Re}\Big[ \frac{1}{2}\sum_{j,k=0}^\infty a_{jk}\big(it-P(z_2)\big)^j z_2^k
+P_{z_2}(z_2) \sum_{m,n=0}^\infty b_{mn} \big(it-P(z_2)\big)^m z_2^n\Big]=0
\end{equation}
for all $z_2\in \mathbb C$ and for all $t\in\mathbb R$ 
with $|z_2|<\epsilon_0$ and $|t|<\delta_0$, where $\epsilon_0>0$ and $\delta_0>0$ 
are small enough. 
The goal is to show that $H\equiv 0$. 
Striving for a contradiction, we suppose that $H\not\equiv 0$. 
Since $P_{z_2}(z_2)$ vanishes to infinite order at $0$, 
we notice that if $h_2\equiv 0$ then (\ref{eq222}) shows that $h_1\equiv 0$. 
So, we must have  $h_2\not\equiv 0$.

We now divide the argument into two cases as follows.

\noindent
{\bf Case 1.} {\boldmath $h_1\not \equiv 0$.} 
In this case let us denote by $j_0$ the smallest integer such that $a_{j_0 k}\ne 0$ 
for some integer $k$. 
Then let $k_0$ be the smallest integer such that $a_{j_0 k_0}\ne 0$. 
Similarly, let $m_0$ be the smallest integer such that $b_{m_0 n}\ne 0$ for some integer $n$. 
Then denote by $n_0$ the smallest integer such that $b_{m_0 n_0}\ne 0$. 
We see that $j_0\geq 1$ if $k_0=0$, and $m_0\geq 1$ if $n_0=0$. 
Since $P(z_2)=o(|z_2|^j)$ for any $j\in \mathbb N$, inserting $t=\alpha P(z_2)$ 
into (\ref{eq223}), where $\alpha\in \mathbb R$ will be chosen later, one has
\begin{equation}\label{eq224}
\begin{split}
\mathrm{Re} \Big[&\frac{1}{2} a_{j_0k_0}(i\alpha -1)^{j_0}(P(z_2))^{j_0}
\big(z_2^{k_0}+o(|z_2|^{k_0})\big) \\
&+ b_{m_0n_0}(i\alpha -1)^{m_0}\big(z_2^{n_0}+o(|z_2|^{n_0})\big) 
(P(z_2))^{m_0}P_{z_2}(z_2)  \Big ]=0
\end{split}
\end{equation}
for all $z_2\in \Delta_{\epsilon_0}$. 
We note that in the case $k_0=0$ and $\mathrm{Re}(a_{j_0 0})=0$, 
$\alpha$ is chosen in such a way that $\mathrm{Re}\big( (i\alpha-1)^{j_0}a_{j_0 0}\big)\ne 0$. 
Then (\ref{eq224}) yields that $j_0>m_0$ by virtue of the fact 
that $P_{z_2}(z_2)$ and $P(z_2)$ vanish to infinite order at $z_2=0$. 
Moreover, we remark that $P_{z_2}(z_2)=\frac{1}{2}P'(x)$, where $x:=\mathrm{Re}(z_2)$. 
Therefore, it follows from (\ref{eq224}) that
\begin{equation}\label{eq225}
\begin{split}
\frac{P'(x)}{\big(P(x)\big)^{j_0-m_0}}
= \frac{\mathrm{Re} \Big[ a_{j_0k_0}(i\alpha -1)^{j_0}\big(z_2^{k_0}+o(|z_2|^{k_0})\big)\Big]}
{\mathrm{Re}\Big[b_{m_0n_0}(i\alpha -1)^{m_0}\big(z_2^{n_0}+o(|z_2|^{n_0})\big)\Big]}
\end{split}
\end{equation}
for all $z_2=x+iy\in \Delta_{\epsilon_0}$ satisfying 
\begin{equation*}
P(x)\ne 0, \; \; 
\mathrm{Re}\Big[b_{m_0n_0}(i\alpha -1)^{m_0}(z_2^{n_0}+o(|z_2|^{n_0}))\Big]\ne 0.
\end{equation*} 
However, (\ref{eq225}) is a contradiction since its right-hand side depends also on $y$, and hence one must have $h_1\equiv 0$.

\noindent
{\bf Case 2.} {\boldmath $h_1\equiv 0$.} 
Let $m_0,n_0$ be as in the Case 1. 
Since $P(z_2)=o(|z_2|^{n_0})$, putting $t=\alpha P(z_2)$ in (\ref{eq223}), 
where $\alpha \in \mathbb R$ will be chosen later, one obtains that
\begin{equation*}
\begin{split}
\frac{1}{2}P'(x)\mathrm{Re} \Big[(i\alpha -1)^{m_0}b_{m_0n_0}
\big(z_2^{n_0}+o(|z_2|^{n_0})\big)  \Big ]=0
\end{split}
\end{equation*}
for all $z_2=x+iy\in \Delta_{\epsilon_0}$. Since $P'(x)\not \equiv 0$, one has
\begin{equation}\label{eq226}
\mathrm{Re} \Big[(i\alpha -1)^{m_0}b_{m_0n_0}\big(z_2^{n_0}+o(|z_2|^{n_0})\big)  \Big ]=0
\end{equation}
for all $z_2\in \Delta_{\epsilon_0}$. Note that if $n_0=0$, then $\alpha$ can be chosen in such a way that 
$\mathrm{Re}\big( (i\alpha -1)^{m_0}b_{m_00} \big)\ne 0$. Hence, (\ref{eq226}) is absurd.

Altogether, the proof of our theorem is complete.
\end{proof}

\end{document}